\documentclass[11pt]{article}

\usepackage{latexsym,color,graphicx,amsmath,amssymb,amsthm}
\usepackage{url}

\newtheorem{theorem}{Theorem}
\newtheorem{lemma}[theorem]{Lemma}
\newtheorem{cor}[theorem]{Corollary}

\def\reals{{\mathbb R}}
\def\R{{\cal R}}
\def\eps{{\varepsilon}}
 
\begin{document}

\title{Sets with few distinct distances do not have heavy lines 
\thanks{%
Work on this paper by Orit E. Raz and Micha Sharir was supported by 
Grant 892/13 from the Israel Science Foundation. Work by Micha Sharir 
was also supported by Grant 2012/229 from the U.S.--Israel Binational 
Science Foundation, by the Israeli Centers of Research Excellence (I-CORE) 
program (Center No.~4/11), and by the Hermann Minkowski-MINERVA Center for 
Geometry at Tel Aviv University. Work by Oliver Roche-Newton was supported 
by the Austrian Science Fund (FWF), Project F5511-N26, which is part of the 
Special Research Program ``Quasi-Monte Carlo Methods: Theory and Applications''.
Part of this research was performed while the authors were visiting
the Institute for Pure and Applied Mathematics (IPAM), which is supported
by the National Science Foundation.}}

\author{
Orit E. Raz\thanks{%
School of Computer Science, Tel Aviv University,
Tel Aviv 69978, Israel.
{\sl oritraz@post.tau.ac.il} }
\and
Oliver Roche-Newton\thanks{%
Johann Radon Institute for Computational and Applied Mathematics,
Linz 4040, Austria.
{\sl o.rochenewton@gmail.com} }
\and
Micha Sharir\thanks{%
School of Computer Science, Tel Aviv University,
Tel Aviv 69978, Israel.
{\sl michas@post.tau.ac.il} }
}

\maketitle

\begin{abstract}
Let $P$ be a set of $n$ points in the plane that determines at most $n/5$ distinct distances.  
We show that no line can contain more than $O(n^{43/52}{\rm polylog}(n))$ points of $P$.
We also show a similar result for rectangular distances, equivalent to distances in the 
Minkowski plane, where the distance between
a pair of points is the area of the axis-parallel rectangle that they span.
\end{abstract}

\section*{The problem and its background}

Given a set $P$ of $n$ points in $\reals^2$, let $D(P)$ denote the number
of distinct distances that are determined by pairs of points from $P$, and put 
$D(n) = \min_{|P|=n}D(P)$; that is, $D(n)$ is the minimum number of distinct 
distances that any set of $n$ points in $\reals^2$ must always determine.
In his celebrated 1946 paper \cite{erd46}, Erd\H os derived the bound 
$D(n) = O(n/\sqrt{\log n})$ by considering a $\sqrt{n}\times \sqrt{n}$ integer lattice.
Recently, after 65 years and a series of progressively larger lower bounds\footnote{%
 For a comprehensive list of the previous bounds, see \cite{GIS11} and 
 \url{http://www.cs.umd.edu/~gasarch/erdos_dist/erdos_dist.html}.}, 
Guth and Katz \cite{GK11} provided an almost matching lower bound $D(n) = \Omega(n/\log n)$.

While the problem of finding the asymptotic value of $D(n)$ is almost completely solved, 
hardly anything is known about which point sets determine a small number of distinct distances. 
Consider a set $P$ of $n$ points in the plane, such that $D(P)=O(n/\sqrt{\log n})$. 
Erd\H os conjectured \cite{erd86} that any such set ``has lattice structure.'' 
Informally, this should mean that, on one hand, there have to exist (many) lines that contain 
many points of $P$, and, on the other hand, no line should contain too many points of $P$. 
Progress on the former aspect of the conjecture has been rather minimal:
The only significant result is a variant of an old proof of Szemer\'edi, 
which implies that there exists a line that contains 
$\Omega(\sqrt{\log n})$ points of $P$ (Szemer\'edi's proof was communicated by 
Erd\H os in \cite{erd75} and can be found in \cite[Theorem 13.7]{PA95}). 

In contrast, some recent works have advanced the latter aspect.
Specifically, a recent result of Pach and de Zeeuw \cite{PdZ13} implies that any constant-degree 
curve that contains no lines and circles cannot be incident to more than $O(n^{3/4})$ 
points of $P$ (see also Sharir et al.~\cite{SSS13} for a precursor of this work). 
Another recent result, by Sheffer, Zahl and de Zeeuw \cite{SZZ13}, 
implies that no line can contain $\Omega(n^{7/8})$ points of $P$, and no circle can 
contain $\Omega(n^{5/6})$ such points.

We note that all these papers do not make use of the specific bound for $D(P)$. What 
they show is that the existence of a curve, line, or circle that contains more than the 
prescribed number of points of $P$ implies that $D(P) = \Omega(n)$. This is also the approach 
used in this paper.

\section*{Our results}

In this paper we significantly improve the bound of Sheffer et al.~\cite{SZZ13} for 
points on a line, and establish the following result (we use the $O^*(\cdot)$ notation 
to hide polylogarithmic factors).


\begin{theorem}\label{main}
Let $P$ be a set of $n$ points in the plane, such that $D(P)\leq n/5$. 
Then, for any line $\ell$ in the plane, we have
$$
|P\cap\ell|=O^*\left(n^{43/52}\right) \approx O(n^{0.827}) .$$
\end{theorem}


We also consider a version of this problem in which a different notion of distance is used. Given two points $p=(p_x,p_y)$ and $q=(q_x,q_y)$ in the plane,
the \emph{rectangular area} or \emph{rectangular distance} determined by $p$ and $q$ is the quantity
$$
R(p,q)=(p_x-q_x)(p_y-q_y).
$$
That is, $R(p,q)$ is the signed area of the axis parallel rectangle with $p$ and $q$ at opposite corners, where the area is positive (resp., negative)
if $p$ and $q$ are the northeast and southwest
(resp., northwest and southeast) corners of the rectangle. Let $\R(P):=|\{R(p,q)\mid p,q \in P\}|$. 
The interest in rectangular distances comes from the observation that, by rotating the coordinate frame by $45^\circ$, 
$R(p,q)$ becomes half the squared distance between $p$ and $q$ in the \emph{Minkowski plane}~\cite{Ben73},
namely, where the squared distance is $(p_x-q_x)^2-(p_y-q_y)^2$. Another motivation for rectangular distances is
that they arise in certain problems of the sum-product type; see Roche-Newton and Rudnev~\cite{RR12}.

Following in the footsteps of the work of Guth and Katz \cite{GK11}, Roche-Newton and Rudnev~\cite{RR12}
have shown that, for a set $P$ of $n$ points in the plane, $\R(P)=\Omega(n/ \log n)$, provided that $P$ 
is not contained inside a single horizontal or vertical line. Note that the latter condition is necessary, since if
all the points lie on such a line then all pairs determine rectangular area equal to zero. Indeed, horizontal and vertical lines are somewhat special 
in this ``metric".

In this context, we can observe a subtle difference between these two notions of distance, since the analogue of Theorem \ref{main} is not quite true if
Euclidean distance is replaced with rectangular area. Let $\eps>0$ be fixed, and suppose for simplicity that 
$n^{1-\eps}$ and $n^{\eps}$ are integers. An unbalanced rectangular lattice 
$\{1,\ldots,n^{1-\eps}\} \times \{1,\ldots,n^{\eps}\}$, consisting of $n$ points, will determine a sublinear number of distinct rectangular
areas. Indeed, the size of the set of rectangular areas is approximately the same as that of the product set $\{1,\ldots,n^{1-\eps}\}\cdot \{1,\ldots,n^{\eps}\}$. 
The fact that
the size of this set is $o(n)$ is a classical result in number theory and precise estimates for its cardinality can be found in Ford \cite{F08}. On the
other hand, this lattice point set contains rich horizontal lines with $n^{1-\eps}$ points thereon. A symmetric construction yields a point set which
determines a sublinear number of rectangular areas, but for which there exist vertical lines which support $n^{1-\eps}$ points. By contrast, it was
established in \cite[Theorem 2.1]{CSS13} that a rectangular lattice needs to be only very slightly unbalanced in order to determine $\Omega(n)$ 
distinct (Euclidean) distances.

However, we show that these are the only problematic directions. Specifically, we prove the following result.

\begin{theorem}\label{main2}
Let $P$ be a set of $n$ points in the plane, such that $\R(P)\leq n/5$. 
Then, for any line $\ell$ which is not horizontal or verical, we have
$$
|P\cap\ell|=O^*\left(n^{43/52}\right) .
$$
\end{theorem}

\section*{Preliminary results}

The proofs of Theorems \ref{main} and \ref{main2} closely follow the structure of the main result in Sheffer et al.~\cite{SZZ13}. 
The quantitative improvements obtained here come as a result of calling upon two results which are quantitatively better than those 
used in \cite{SZZ13}. The first of these is the following incidence theorem.

\begin{theorem}[{\bf Agarwal et al., Theorem 6.6~\cite{lenses}}]\label{lenses}
Let $C$ be a family of distinct pseudo-parabolas that admit a $3$-parameter representation, 
and let $P$ be a set of distinct points in the plane. 
Denote by $I(P,C)$ the number of incidences between $P$ and $C$. Then
$$
I(P,C)=O(|P|^{2/3}|C|^{2/3}+|P|^{6/11}|C|^{9/11}\log^{2/11}|C|+|P|+|C|).
$$
\end{theorem}
(The bound in \cite{lenses} is slightly weaker; the improvement, manifested in the factor $\log^{2/11}n$, 
which replaces a slightly larger factor in \cite{lenses}, is due to Marcus and Tardos~\cite{MT}.)
In the language of \cite{lenses}, a family of \textit{psuedo-parabolas} is a family of graphs of everywhere defined continuous functions, so that each
pair intersect in at most two points. The collection $C$ \textit{admits a $3$-parameter representation} if the curves have three degrees of
freedom, and can thus be identified with points in $\reals^3$ in a suitable manner. 
A full definition of this property is given at the beginning of Section 5 in \cite{lenses}. 

The important observation for us is that a
family of hyperbolas of the form $(x-a)^2-(y-b)^2=c$, where $a,b$ and $c$ are real numbers with $c \neq 0$, is a family of pseudo-parabolas which admit a
$3$-parameter representation, and so the bound in Theorem \ref{lenses} applies to such a family.
(Technically, some transformations are needed to make this family have the desired properties. Specifically, we rotate the coordinate frame by $90^\circ$, 
and treat each branch of each hyperbola as a separate curve.) Such families of hyperbolas arise in the proofs of our theorems.

A crucial assumption in Theorem~\ref{lenses} is that the curves in $C$ are all \emph{distinct}.  
Suppose next that they are not necessarily distinct, but that the maximum coincidence multiplicity 
of any curve in $C$ is at most $k$. A standard argument,
similar to the one used by Sz\'ekely~\cite{Sz}, shows that in this case we have
\begin{equation} \label{incmult0}
I(P,C) = O\left( k^{1/3}|P|^{2/3}|C|^{2/3} + k^{2/11}|P|^{6/11}|C|^{9/11}\log^{2/11}|C| +
k|P| + |C| \right).
\end{equation}

One way of seeing this is to use a pruning argument, where we leave just one
curve out of any family of coinciding ones, assuming that all multiplicities 
are between $t$ and $t/2$ for some $t$. This leaves us with a subset $C'$ of 
$\Theta(|C|/t)$ curves, all distinct. Applying Theorem~\ref{lenses} to 
$P$ and $C'$, and multiplying the resulting bound by $t$, the asserted bound 
follows, with $t$ instead of $k$, and with $C$ standing for the subset of curves
whose multiplicity is roughly $t$.
To complete the analysis, we sum the bounds over the geometric sequence of
ranges of $t$, up to $k$.

The other result that will be called upon, which is a sharper variant of a result in \cite{LR}, is the following.

\begin{theorem}
\label{LR}
Let $f$ be a continuous strictly convex or concave function on $\reals$, and let
$U,V\subset\reals$ be finite sets. Then
\begin{equation}\label{eq:LR}
|U-U|^5|f(U)+V|^6=\Omega\left( {\frac{|U|^{11}|V|^3}{\log^2|U|}} \right).
\end{equation}
\end{theorem}

This result represents a quantitative improvement on an earlier work Elekes, Nathanson and Ruzsa \cite{ENR99}, which was used in the work of \cite{SZZ13}.
As noted, this was not the precise form in which the bound originally appeared. In \cite{LR}, it was assumed that the sets $U$ and $V$ were of
comparable size, and so the numerator on the right-hand side of \eqref{eq:LR} was simply written as $|U|^{14}$. However this assumption was in fact
completely unnecessary. By working through the original proof without this needless assumption, the result is Theorem \ref{LR}. Since the exact result that
will be used has not appeared in the literature, the proof is included here as an appendix.


\section*{Proof of Theorem \ref{main}}

Let $\ell$ be a line for which $|P\cap\ell|$ is maximal, denote this value as $m$,
and put $A:=P\cap\ell$. Assume for simplicity that $\ell$ is the $x$-axis.
We face a setup where we have two sets $A$ and $P$, where $A$ consists 
of $m$ points on the $x$-axis.
We want to show that $m$ cannot be too large, given our assumption that the number of distinct
distances in $P$ is at most $n/5$. The strategy is to show that if $m$ is too large then either
$D(A)$, the number of distinct distances in $A$, or $D(A,P)$, the number of 
distinct distances in $A\times P$, is large. The main part of the analysis focuses on the latter quantity.

Following the by now usual strategy, as applied in several recent works (e.g., see 
\cite{PdZ13,SSS13,SS13}), we consider the set of quadruples
$$
Q := \{(a,b,p,q)\in A^2\times P^2 \mid \|p-a\|=\|q-b\| \} ,
$$ 
and double count its cardinality. A lower bound is easy to obtain, via the
Cauchy-Schwarz inequality. That is, enumerate the $D(A,P)$ distinct distances
in $A\times P$ as $\delta_1,\ldots,\delta_{D(A,P)}$, write, for each $i$,
$$
M_i = \left| \{ (a,p) \in A\times P \mid \|p-a\|=\delta_i \} \right| ,
$$
and note that
$$
mn = \sum_{i=1}^{D(A,P)} M_i \le \left( \sum_i M_i^2 \right)^{1/2} D(A,P)^{1/2}
= |Q|^{1/2}D(A,P)^{1/2} ,
$$
or, since $D(A,P)\leq n/5$,
\begin{equation}\label{11}
|Q| \ge \frac{m^2n^2}{D(A,P)}\geq 5m^2n .
\end{equation}
We partition $Q$ into two parts: $Q^{(1)}$ contains the quadruples $(a,b,p,q)\in Q$ 
for which $p_y^2 = q_y^2$ and $Q^{(2)} = Q \setminus Q^{(1)}$. We first bound $|Q^{(1)}|$,
noting that for any choice of the points $a, b, p$, there are at most four choices of $q$ 
such that $(a,b,p,q)\in Q^{(1)}$. Hence, $|Q^{(1)}|\leq 4m^2n$. Using (\ref{11}), we get
\begin{equation} \label{Q2}
|Q^{(2)}| \geq m^2n . 
\end{equation}

For an upper bound on $|Q^{(2)}|$, we again follow the standard approach.
That is, we map each pair $(p,q)\in P^2$, with $p_y^2\ne q_y^2$, into the curve 
$$
\gamma_{p,q} = \{ (x,y) \in\reals^2 \mid \|p-(x,0)\|^2 = \|q-(y,0)\|^2 \} ,
$$
and observe that $|Q^{(2)}|$ is equal to the number of incidences between the curves
$\gamma_{p,q}$ and the points of $\Pi:=A^2$, where each ordered pair of points in $A$ is 
interpreted as (the $x$- and $y$-coordinates of) a point in a suitable parametric 
plane.

The curves are in fact hyperbolas. Specifically, the equation of $\gamma_{p,q}$,
for $p=(p_x,p_y)$ and $q=(q_x,q_y)$, is
\begin{align} \label{hypeq}
(x-p_x)^2 + p_y^2 & = (y-q_x)^2 + q_y^2 , \quad\quad\text{or} \nonumber \\
x^2 - y^2 - 2p_xx + 2q_xy & = q_x^2 + q_y^2 - p_x^2 - p_y^2 .
\end{align}

The requirement that $p_y^2\ne q_y^2$ ensures that $\gamma_{p,q}$ is a
non-degenerate hyperbola (i.e., not the union of two lines); in particular,
the quadratic polynomial defining $\gamma_{p,q}$ is irreducible.
We also note that each hyperbola has three degrees of freedom (as required in 
Theorem~\ref{lenses})---it can be specified
by the parameters $p_x$, $q_x$, and $q_y^2-p_y^2$.

Let $\Gamma$ denote the multiset of these hyperbolas. As in the previously
cited works, the main difficulty in the analysis is the possibility that
many hyperbolas in $\Gamma$ coincide, in which case the known machinery for
deriving incidence bounds (such as the bound in Theorem~\ref{lenses})
breaks down, and the bounds themselves become too weak. 
Indeed, the hyperbolas might coincide; let $k$ denote the maximum coincidence 
multiplicity of any hyperbola in $\Gamma$. Then, repeating the bound in (\ref{incmult0}), we have
\begin{equation} \label{incmult}
|Q^{(2)}| = I(\Pi,\Gamma) = O\left( k^{1/3}|\Pi|^{2/3}|\Gamma|^{2/3} +
k^{2/11}|\Pi|^{6/11}|\Gamma|^{9/11}\log^{2/11}|\Gamma| +
k|\Pi| + |\Gamma| \right).
\end{equation} 

\paragraph{The multiplicity of hyperbolas.}

The next step applies an argument of the sum-product type, based on Theorem~\ref{LR}, to obtain an upper 
bound on $k$. Let $\gamma_{p,q}$ and $\gamma_{p',q'}$ be two coinciding hyperbolas.
By (\ref{hypeq}), we have $p_x=p'_x$, $q_x=q'_x$, and
$q_y^2-p_y^2 = (q'_y)^2-(p'_y)^2$. In other words, $k$ coinciding hyperbolas
$\gamma_{p_i,q_i}$, for $i=1,\ldots,k$, are such that all the points $p_i$
lie on the same vertical line, all the points $q_i$ lie on another common 
vertical line, and the quantities $q_{iy}^2-p_{iy}^2$ are all equal.
The points $p_i$ need not be distinct, but each point can appear at most twice.
Indeed, the condition $\gamma_{p,q} = \gamma_{p,q'}$ is equivalent to $q_x = q'_x$ and
$q_y^2 = (q'_y)^2$, so, for a given $q$, there is at most one choice for another
point $q'$. That is, in this situation we have at least $k/2$ distinct points $p_i$, and
at least $k/2$ distinct points $q_i$.

Consider the following restricted problem. We have a set $A$ of $m$ points 
on the $x$-axis, and (with a slight abuse of notation)
a set $B$ of at least $k/2$ points on some vertical line (say, we
take $B$ to be the set of the points $q_i$ defining our $k$ coinciding hyperbolas), 
which we may assume to be the $y$-axis.  Consider the sets
\begin{align*}
A - A & = \{x-y \mid x,y\in A \} \\
A^2 + B^2 & = \{x^2+y^2 \mid x\in A,\; y\in B\} .
\end{align*}

Notice that $|A-A|$ is at most twice the number of distinct distances in $A$, and
that $|A^2+B^2|$ is the number of distinct distances in $A\times B$.  By the hypothesis 
of the theorem, we have that both $|A-A|$ and $|A^2+B^2|$ are $O(n)$.

We apply Theorem~\ref{LR} with $U=A$, $V=B^2$, and $f(x)=x^2$, and obtain
$$
|A-A|^5|f(A)+B^2|^6=
|A-A|^5|A^2+B^2|^6=
\Omega^*\left(|A|^{11}|B|^{3}\right) ;
$$
that is, 
$$
|A|^{11}|B|^{3} = O^*(n^{11}) , \quad\quad\text{or}\quad\quad 
|B| = O^*\left(n^{11/3}/|A|^{11/3}\right) = O^*\left(n^{11/3}/m^{11/3}\right) .
$$
In other words, we have shown that no vertical line can contain more than 
$\nu = O^*(n^{11/3}/m^{11/3})$ points. In particular, we get that the multiplicity 
bound $k$ that we are after satisfies $k \le 2\nu = O^*(n^{11/3}/m^{11/3})$. 
Substituting this into (\ref{incmult}),
noting that $|\Gamma| \le n^2$ and $|\Pi| \le m^2$, we obtain
\begin{align*}
|Q^{(2)}| & = I(\Pi,\Gamma) \\
& = O^*\left( k^{1/3}|\Pi|^{2/3}|\Gamma|^{2/3} +
k^{2/11}|\Pi|^{6/11}|\Gamma|^{9/11} + k|\Pi| + |\Gamma| \right) \\
& = O^*\left( (n^{11/3}/m^{11/3})^{1/3}(m^2)^{2/3}(n^2)^{2/3} +
(n^{11/3}/m^{11/3})^{2/11}(m^2)^{6/11}(n^2)^{9/11} \right. \\
& \left. \quad\quad\quad + n^{11/3}/m^{5/3} + n^2 \right) \\
& = O\left( m^{1/9}n^{23/9} + m^{14/33}n^{76/33} + n^{11/3}/m^{5/3} \right) 
\end{align*}
(the term $n^2$ is subsumed by the first two terms).
Recalling that $|Q^{(2)}| \ge m^2n$, we have
$$
m^2n = O^*\left( m^{1/9}n^{23/9} + m^{14/33}n^{76/33} + n^{11/3}/m^{5/3} \right) ,
$$
and an easy calculation shows that $m = O^*(n^{43/52})$.
This completes the proof of Theorem~\ref{main}. 
\qed

\noindent{\bf Remark.}
The bound on $m$ becomes in fact even (slightly) smaller than $n^{43/52}$ when
$D(P) = o\left( n/\log^{2/11}n\right)$. For example, in the extreme case where
$D(P) = O\left( n/\sqrt{\log n}\right)$, we get $m = O\left(n^{43/52}/\log^{21/52}n\right)$.


\section*{Proof of Theorem \ref{main2}}
Let $\ell$ be a line in a direction which is not horizonal or vertical and with the property that $|P\cap \ell|$ is 
maximal over all lines of this type, and put $m:=|P\cap\ell|$.
Write $A:=P \cap \ell$. The proof follows 
the same structure as that of the proof of Theorem \ref{main}, and shows that $m=|A|=O^*(n^{43/52})$.

Since the quantity $\R(P)$ is invariant under translation of the set $P$, we may assume that $\ell$ passes 
through the origin.\footnote{%
  Note that $\R(P)$ is not invariant under rotation, which is why we cannot assume that the line $\ell$ is the 
  $x$-axis, as we did in the proof of Theorem \ref{main}.}
Furthermore, since we assume that $\ell$ is not horizontal or vertical, its equation is $y=\kappa x$, 
for some $\kappa  \neq 0$. It will be convenient to use the notation $A_0 :=\{a\mid (a,\kappa a) \in A\}$
for the set of $x$-coordinates of the points in $A$.

Let $Q$ denote the set of all quadruples $(s,t,p,q) \in A^2 \times P^2$ which satisfy the equation
\begin{equation}
R(p,s)=R(q,t) .
\label{energyeqn}
\end{equation}
By the Cauchy-Schwarz inequality and the conditions of the theorem, we have, similar to the proof of Theorem~\ref{main},
\begin{equation}
\frac{n|Q|}{5} \geq \R(A,P)|Q|\geq |A|^2n^2 = m^2n^2 ,
\label{CSstuff}
\end{equation}
where $\R(A,P)$ denotes the number of distinct rectangular distances between the points of $A$ and the points of $P$. Hence
\begin{equation}
|Q| \geq 5m^2n.
\label{CSstufff}
\end{equation} 

A satisfactory upper bound on $|Q|$ will be sufficient to prove that $\R(P)$ is large. 
Once again, the set of quadruples $Q$ is partitioned into two subsets: $Q^{(1)}$ contains 
the set of all quadruples $(s,t,p,q) \in Q$ such that $(q_y-\kappa q_x)^2 = (p_y-\kappa p_x)^2$, and 
$Q^{(2)}:=Q \setminus Q^{(1)}$. Geometrically, in the quadruples of $Q^{(1)}$, $p$ and $q$ lie at the same (absolute) distance from $\ell$.
For each triple $(s,t,p) \in A^2\times B$, there exists at most four choices of $q=(q_x,q_y)$ that satisfy 
both \eqref{energyeqn} and $(q_y-\kappa q_x)^2 = (p_y-\kappa p_x)^2$, as is easily checked.
Therefore, $|Q^{(1)}|\leq 4m^2n$, and it thus follows from \eqref{CSstufff} that
\begin{equation}
|Q^{(2)}| \geq m^2n.
\label{CSstuff2}
\end{equation}

Following the proof of Theorem~\ref{main}, the task of obtaining an upper bound for $|Q^{(2)}|$ 
can be reformulated as an incidence problem. For $p,q \in P$, write $p=(p_x,p_y)$ and $q=(q_x,q_y)$, 
and define $\gamma_{p,q}$ to be the curve given by the equation
$$
(p_x-x)(p_y-\kappa x)=(q_x-y)(q_y-\kappa y).
$$
After expanding and rearranging, it follows that $\gamma_{p,q}$ is the hyperbola with equation
\begin{align}
\left(x-\frac{p_y+\kappa p_x}{2\kappa }\right)^2-\left(y-\frac{q_y+\kappa q_x}{2\kappa }\right)^2
& = \frac{q_xq_y-p_xp_y}{\kappa} +
\frac{ (p_y+\kappa p_x)^2 - (q_y+\kappa q_x)^2 }{4\kappa^2} \nonumber \\
& = \frac{ (p_y-\kappa p_x)^2 - (q_y-\kappa q_x)^2 }{4\kappa^2} .
\label{Spqeq}
\end{align}
Let $\Gamma$ be the multiset of curves
$$
\Gamma :=\{\gamma_{p,q} \mid p,q\in{P},\; (p_y-\kappa p_x)^2 \ne (q_y-\kappa q_x)^2 \} .
$$
The curves of $\Gamma$ are non-degenerate hyperbolas.
As in the proof of Theorem~\ref{main}, $\Gamma$ is a family of pseudo-parabolas that admits a 3-parameter representation,
except that the curves in $\Gamma$ can occur with multiplicity. Concretely, $\gamma_{p,q}=\gamma_{p',q'}$ if each of
the following conditions holds.
\begin{align}
p_y+\kappa p_x&=p_y'+\kappa p_x' \label{cond1} \\
q_y+\kappa q_x &= q_y'+\kappa q_x' \label{cond2} \\
(p_y-\kappa p_x)^2 - (q_y-\kappa q_x)^2 & =
(p'_y-\kappa p'_x)^2 - (q'_y-\kappa q'_x)^2 . \label{cond3}
\end{align}

Recalling the notation $A_0=\{a \mid (a,\kappa a) \in A\}$, we put $\Pi:=A_0\times A_0$. 
It is immediate from the definitions that $|Q^{(2)}| = I(\Pi,\Gamma)$.
Let $k$ denote the maximum multiplicity of a hyperbola in $\Gamma$. 
Since $\Gamma$ satisfies the assumptions of Theorem~\ref{lenses},
inequality \eqref{incmult0} holds for these curves as well, giving
\begin{equation} \label{incmult2}
I(\Pi,\Gamma) = O\left( k^{1/3}|\Pi|^{2/3}|\Gamma|^{2/3} +
k^{2/11}|\Pi|^{6/11}|\Gamma|^{9/11}\log^{2/11}|\Gamma| +
k|\Pi| + |\Gamma| \right).
\end{equation} 

The next step is to show that $k$ cannot be too large, in order to make use of \eqref{incmult2}. 
Once again, this step follows the same argument as the corresponding part of the proof of 
Theorem \ref{main}. The only difference is that a different convex function (in the application
of Theorem~\ref{LR}) will be used. Specifically, we claim:

\begin{equation} \label{claim2}
k=O^*\left(\frac{n^{11/3}}{m^{11/3}}\right).
\end{equation}

To prove \eqref{claim2}, first suppose that some curve in $\Gamma$ has multiplicity $k$. That is, there exist $k$ pairs of points 
$(p^{(1)},q^{(1)}), (p^{(2)},q^{(2)}),\ldots,(p^{(k)},q^{(k)})$ in $B^2$,
where $p^{(i)}=(p_x^{(i)},p_y^{(i)})$ and $q^{(i)}=(q_x^{(i)},q_y^{(i)})$, such that 
$\gamma_{p^{(1)},q^{(1)}}=\gamma_{p^{(2)},q^{(2)}}=\cdots=\gamma_{p^{(k)},q^{(k)}}$. In particular, by \eqref{cond1},
$$
p_y^{(1)}+\kappa p_x^{(1)} = p_y^{(2)}+\kappa p_x^{(2)} = \cdots = p_y^{(k)}+\kappa p_x^{(k)}.
$$
Denoting this common value as $c$, it follows that all the points $p^{(i)}$ lie 
on a common line $\ell_2$, given by $y=-\kappa x + c$.

As in the preceding analysis, the points $p^{(i)}$ need not be distinct. Nevertheless, a point $p$
can have at most two points $q^{(1)}$, $q^{(2)}$ such that $\gamma_{p,q^{(1)}} = \gamma_{p,q^{(2)}}$.
Indeed, writining $p=(p_x,p_y)$, $q^{(1)}=(q^{(1)}_x,q^{(1)}_y)$, $q^{(2)}=(q^{(2)}_x,q^{(2)}_y)$,
the above coincidence of hyperbolas is equivalent to
\begin{equation}
(p_y-\kappa p_x)^2 - (q^{(1)}_y-\kappa q^{(1)}_x)^2 =
(p_y-\kappa p_x)^2 - (q^{(2)}_y-\kappa q^{(2)}_x)^2 ,
\label{impossible1}
\end{equation}
and
\begin{equation}
q_y^{(1)}+\kappa q_x^{(1)}=q_y^{(2)}+\kappa q_x^{(2)} .
\label{impossible2}
\end{equation}
Fixing $q^{(1)}$, this gives a system of a linear equation and a quadratic equation
in the coordinates of $q^{(2)}$, which has at most two solutions, one of which is $q^{(1)}$
itself.

Hence, $\ell_2$ contains at least $k/2$ \emph{distinct} points of $B$. 
Write $B_0:=\{x \mid (x,y) \in \ell_2 \cap B\}$ for the set of $x$-coordinates of 
$\ell_2 \cap B$, so that $|B_0|=|\ell_2 \cap B|\geq k/2$.

Observe that
\begin{equation}
\R(A)=|\{(a-b)(\kappa a-\kappa b) \mid a,b \in A_0\}| = 
|\{(a-b)^2 \mid a,b \in A_0\}| \ge \frac{|A_0-A_0|}{2}.
\label{A-A}
\end{equation}
Finally, we have
\begin{align*}
\R(A,B) &\geq \R(A, B \cap \ell_2) \\
& = |\{(b-a)(-\kappa b+c-\kappa a) \mid  a \in A_0,b\in B_0\}| \\
& = \left|\left\{\left(\kappa a^2-ca\right)-\left(\kappa b^2-cb\right) \mid a \in A_0, b\in B_0\right\}\right|
\\&=|f(A_0)-f(B_0)|,
\end{align*}
where $f(x)=\kappa x^2 - cx$. 
Since $f$ is a strictly convex function we can apply Theorem \ref{LR} with 
$U=A_0$ and $V=B_0$ to deduce that
\begin{align*}
\R(P)^{11} &\geq \R(A)^5\R(A,B)^6
\\&\geq \frac12 |A_0-A_0|^5|f(A_0)-f(B_0)|^6
\\&=\Omega^* (|A_0|^{11}|B_0|^3).
\end{align*}
Applying the hypothesis that $\R(P)=O(n)$, it follows that 
$$
k \le 2|B_0|=O^*\left(\frac{n^{11/3}}{|A|^{11/3}}\right) = 
O^*\left(\frac{n^{11/3}}{m^{11/3}}\right) , 
$$
which establishes \eqref{claim2}.

Combining \eqref{claim2} with \eqref{incmult2}, it follows that
\begin{align*}
|Q^{(2)}|&=O^*(k^{1/3}m^{4/3}n^{4/3}+k^{2/11}m^{12/11}n^{18/11}+km^2+ n^2)
\\&=O^*(m^{1/9}n^{23/9}+m^{14/33}n^{76/33}+n^{11/3}/m^{5/3}).
\end{align*}
Recalling from \eqref{CSstuff2} that $|Q^{(2)}| \ge m^2n$, we have
$$
m^2n=O^*(m^{1/9}n^{23/9}+m^{14/33}n^{76/33}+n^{11/3}/m^{5/3}),
$$
which implies, as before, that $|A|=m=O^*(n^{43/52})$, as required. \qed

\appendix

\section*{Appendix: Convexity and sumsets: Unbalanced version}

In this appendix we give the proof of Theorem~\ref{LR}, by spelling out\footnote{The proofs in the appendix largely follow the work of \cite{LR}. The only exception is Lemma \ref{lemma:main}, the proof of which has been simplified from its original presentation. We are grateful to Misha Rudnev for explaining this simplification.}
the details of the analysis in Li and Roche-Newton~\cite{LR}, adapted to the
unbalanced case in which $|U|$ and $|V|$ are not necessarily comparable.

\subsection*{Preliminary results}

One of the tools needed for the proof of Theorem \ref{LR} is a generalization 
of the notion of the energy of a set. The \textit{additive energy} of a finite set 
$A\subset\reals$, denoted $E_2(A)$, is the number of solutions to the equation
$$
a-b=c-d,\quad\quad a,b,c,d \in A.
$$
This quantity can be rewritten as
$$
E_2(A)=\sum_x r_{A-A}^2(x),
$$
where $r_{A-A}(x):=|\{(a,b) \in A \times A \mid a-b=x\}|$ is the number of 
representations
of $x$ in $A-A$. Similarly, define, for any positive rational $k$,
\begin{align*}
E_k(A) & := \sum_x r_{A-A}^k(x) , \quad\text{and} \\
E_k(A,B) & := \sum_x r_{A-B}^k(x) .
\end{align*}
For $k=2$ we drop the index, and write $E(A)$ for $E_2(A)$ and $E(A,B)$ for $E_2(A,B)$.

We need the following result, which was originally established in Li~\cite[Lemma 2.4, Lemma 2.5]{Li}.
\begin{lemma}\label{E1.5}
Let $A,B$ be finite subsets of a field $\mathbb F$. Then
$$
E_{1.5}(A)^2\cdot|B|^2\leq E_3(A)^{2/3}\cdot E_3(B)^{1/3}\cdot E(A,A-B).
$$
\end{lemma}

The other tool that will be used is the following well-known variant of the Szemer\'{e}di--Trotter Theorem 
for pseudo-lines (or pseudo-segments), which is a special case of \cite[Theorem 8]{Sz}.
\begin{lemma} \label{thm:STcurvy}
Let $P$ be a finite set of points and let $L$ be a finite family of simple\footnote{%
  A planar curve is said to be {\em simple} if it is an injective image of a continuous map from $[0,1]$ into $\reals$.}
curves in $\mathbb R^2$ with the property that any pair of curves intersect in at most one point. Then
$$
I(P,L) = O\left( |P|^{2/3}|L|^{2/3} +|P|+|L| \right).
$$
In particular, it follows that, for any integer $t\geq 2$, the set 
$P_t$ of all points $p \in \mathbb R^2$ that are incident to at least $t$ lines of $L$ satisfies
$$
|P_t| = O\left( \frac{|L|^2}{t^3}+ \frac{|L|}{t} \right).
$$
\end{lemma}

\subsection*{Energy bounds}
\begin{lemma}\label{lemma:main}
Let $f$ be a continuous, strictly convex (or concave) function on the reals, and 
$A,B,C\subset\mathbb{R}$ be finite sets. Then, for all $0< t \le \min\{|A|,|B|\}$, 
\begin{equation}
\label{ST2}
\big|\{x\in\reals \mid r_{A-B}(x)\geq t\}\big| = O\left( \frac{|f(A)+C|^2|B|^2}{|C|t^3} \right).
\end{equation}
\end{lemma}

\begin{proof} 
Define $l_{s,b}$ to be the curve with equation $y=-f(x+b)+s$, and put
$$
L:=\{l_{s,b} \mid (s,b)\in (f(A)+C) \times B\}.
$$
Note that $|L|=|f(A)+C||B|$. Since the lemma is trivially true for $t<2$ (because the left-hand 
side is at most $|A-B|\le|A|\cdot|B|$, which is dominated by the right-hand side), 
we may assume that $t\geq 2$. Let $P_t$ denote the set of points in $\mathbb R^2$ that are 
incident to at least $t$ curves from $L$. The curves in $L$ satisfy the conditions of 
Lemma \ref{thm:STcurvy}. Indeed, let $(b,s)\neq (b',s')\in \reals^2$. If $b=b'$, then clearly $l_{s,b}\cap l_{s',b'}$ is empty. 
Otherwise, assume without loss of generality that $b'>b$, and put $h(x):=f(x+b')-f(x+b)+s-s'$. 
A point $(x,y)\in l_{s,b}\cap l_{s',b'}$ satisfies $h(x)=0$. 
It is a consequence of the convexity of $f$ that $h$ is a monotone function, and the equation 
$h(x)=0$ has therefore at most one solution. Therefore, Lemma \ref{thm:STcurvy} implies that
\begin{equation}
|P_t| = O\left( \frac{|f(A)+C|^2|B|^2}{t^3}+ \frac{|f(A)+C||B|}{t} \right).
\label{STrich}
\end{equation}
We may assume that the first term in \eqref{STrich} is greater than or equal to one quarter of
the second term. Indeed, otherwise we would have 
$$
t > 2 |f(A)+C|^{1/2}|B|^{1/2} \geq 2 |f(A)|^{1/2}|B|^{1/2}\geq \min\left\{|A|,2|B|\right\}.
$$ 
Since there does not exist an $x$ with $r_{A-B}(x) > \min\{|A|,|B|\}$, the left-hand side of
\eqref{ST2} is $0$ in this case, making the bound trivial. We can therefore assume that 
\begin{equation}
|P_t| =O\left( \frac{|f(A)+C|^2|B|^2}{t^3} \right).
\label{STrich2}
\end{equation}
Now, take any $x\in\reals$ that satisfies $r_{A-B}(x) \geq t$. 
Then there exist $(a_1,b_1),\ldots,(a_t,b_t) \in A \times B$, such that
$x=a_1-b_1=\cdots=a_t-b_t$.
It follows that, for any $c\in C$, and $1 \leq i \leq t$,
$$
c=-f(x+b_i)+f(a_i)+c,
$$
which implies that $(x,c) \in l_{f(a_i)+c,b_i}$ for all $1 \leq i \leq t$. 
Therefore $(x,c) \in P_t$ and so
$$
|C|\cdot |\{x \mid r_{A-B}(x) \geq t\}| \leq |P_t| = O\left( \frac{|f(A)+C|^2|B|^2}{t^3} \right) ,
$$
which completes the proof.
\end{proof}

By applying Lemma \ref{lemma:main} carefully, we obtain the following corollary.

\begin{cor}\label{E3A}
Let $f$ be a continuous, strictly convex or concave function on the reals, 
and let $A,C,F\subset\mathbb{R}$ be finite sets.
Then
\begin{align}
E_3(A)& = O\left( |f(A)+C|^2|A|^2|C|^{-1}\log|A| \right), \label{1111}\\
E(A,F)& = O\left( |f(A)+C||F|^{3/2}|A|^{1/2}|C|^{-1/2} \right). \label{1010}
\end{align}
\end{cor}

\begin{proof} 
Applying Lemma \ref{lemma:main} with $B=A$, gives
\begin{align*}
E_3(A)& = \sum_{j=0}^{\lfloor\log|A|\rfloor}\sum_{\{s \;\mid\; 2^j\leq r_{A-A}(s)<2^{j+1}\}}r^3_{A-A}(s)\\
& = O\left( \sum_{j=0}^{\lfloor\log|A|\rfloor}\frac{|f(A)+C|^2|A|^2}{|C|2^{3j}}\cdot2^{3j} \right)
= O\left( \frac{|f(A)+C|^2|A|^2\log|A|}{|C|} \right) ,
\end{align*}
which proves (\ref{1111}). 

Similarly, applying Lemma \ref{lemma:main} with $B=F$ and with any fixed real parameter
$\triangle\ge 1$, gives
\begin{align*}
E(A,F)& = \sum_{\{s \;\mid\; r_{A-F}(s)<{\triangle}\}}r^2_{A-F}(s)+
\sum_{j=0}^{\lfloor\log(|A|/\triangle)\rfloor}\sum_{\{s \;\mid\; 2^j\triangle\leq r_{A-F}(s)<2^{j+1}\triangle\}}r^2_{A-F}(s) \\
& = O\left(\triangle\cdot E_{1}(A,F)+
\sum_{j=0}^{\lfloor\log(|A|/\triangle)\rfloor}\frac{|f(A)+C|^2|F|^2}{|C|2^{3j}\triangle^{3}}\cdot2^{2j}\triangle^{2} 
\right) \\
& = O\left( \triangle|A||F|+\frac{|f(A)+C|^2|F|^2}{|C|\triangle} \right) .
\end{align*}
Choosing $\triangle=\frac{|f(A)+C||F|^{1/2}}{|A|^{1/2}|C|^{1/2}}$, which is clearly $\ge 1$, yields (\ref{1010}). 
\end{proof}

\subsection*{Proof of Theorem \ref{LR}} 

First, apply H\"{o}lder's inequality to bound $E_{1.5}(U)$ from below, as follows.
$$
|U|^6=\left(\sum_{s\in{U-U}}r_{U-U}(s)\right)^3
\leq {\left(\sum_{s\in{U-U}}r^{1.5}_{U-U}(s)\right)^2|U-U|}=E_{1.5}(U)^2|U-U|.
$$
Therefore, using the above bound and Lemma \ref{E1.5} with $A=B=U$,
gives
\[
\frac{|U|^8}{|U-U|}\leq{E_{1.5}(U)^2|U|^2}\leq{E_3(U)E(U,U-U)}.
\]
Finally, apply (\ref{1111}) and (\ref{1010}), with $A=U$, $C=V$, and $F=U-U$, 
to conclude that
\[
\frac{|U|^8}{|U-U|} = O\left({|f(U)+V|^3|U-U|^{3/2}|U|^{5/2}|V|^{-3/2}\log|U|} \right),
\] 
and hence
\[
|f(U)+V|^6|U-U|^5 = \Omega\left( {\frac{|U|^{11}|V|^3}{\log^2|U|}} \right),
\] 
as required. \qed

\subsection*{Acknowledgments.}
We acknowledge, with thanks, discussions of the ideas in this paper
with Adam Sheffer, Josh Zahl and Frank de Zeeuw, the authors of \cite{SZZ13}.

\end{document}